\documentclass[a4paper,10pt]{article}

\usepackage{amssymb}
\usepackage[frenchb]{babel}	
\usepackage{fancyhdr}		
\usepackage[frenchb]{babel}	
\newcommand{\fin}{\hfill$\square$}

\usepackage[T1]{fontenc}
\usepackage[latin1]{inputenc}
\usepackage[a4paper,textwidth=14cm,textheight=22cm]{geometry}
\usepackage{url}	
\usepackage{graphicx}	
\usepackage[frenchb]{babel}
\usepackage{amssymb}
\usepackage{amsmath}
\usepackage{amsthm}
\usepackage[T1]{fontenc}
\usepackage{ae}

\newtheorem{theo}{Théorème}[section]
\newtheorem{cor}[theo]{Corollaire}
\newtheorem{lemm}[theo]{Lemme}

\newtheorem{defi}[theo]{Définition}

\newtheorem{prop}[theo]{Proposition}

\newtheorem{rmq}[theo]{Remarque}

\title{Sur la géométrie de la singularité initiale
des espaces-temps plats globalement hyperboliques}
\author{Mehdi Belraouti}

\begin{document}

\maketitle

\noindent{\bf Résumé}
On étudie le comportement asymptotique des niveaux d'une fonction temps quasi-concave, définie sur un espace-temps globalement hyperbolique maximal plat de dimension trois, admettant une hypersurface de Cauchy de genre $\geq 2$. On donne une réponse positive à une conjecture posée par Benedetti et Guadagnini dans \cite{MR1857817}. Plus précisément, on montre que les niveaux d'une telle fonction temps convergent au sens de la topologie de Hausdorff-Gromov équivariante vers un arbre réel. On montre de plus que la limite est indépendante de la fonction temps choisie.

\vskip 0.5\baselineskip
\noindent{\bf Abstract}
Let $M$ be a maximal globally hyperbolic Cauchy compact flat spacetime of dimension $2+1$, admitting a Cauchy hypersurface diffeomorphic to a compact hyperbolic manifold. 
We study the asymptotic behaviour of level sets of quasi-concave time functions on $M$. We give a positive answer to a conjecture of Benedetti and Guadagnini in \cite{MR1857817}. More precisely, we prove that the level sets of such a time function converge in the Hausdorff-Gromov equivariant topology to a real tree. Moreover, this limit does not depend on the choice of the time function.

\section{Introduction}

Un espace-temps est une variété lorentzienne orientée chronologiquement orientée. Il est dit globalement hyperbolique et on écrit $GH$, s'il admet une hypersurface de Cauchy i.e une hypersurface spatiale plongée, rencontrant toutes les courbes causales inextensibles une et une seule fois. Si de plus elle est compacte, alors toutes les autres hypersurfaces de Cauchy le sont et dans ce cas on dit que l'espace-temps $GH$ est Cauchy compact et on écrit $GHC$. Il est dit plat s'il est localement isométrique à l'espace de Minkowski i.e muni d'une $(G,X)$ structure où $X$ est l'espace de Minkowski de dimension $n+1$ (not\'e $\mathbb{R}^{1,n}$) et $G$ le groupe de Poincar\'e, \textit{i.e.} le groupe des isom\'etries affines de $\mathbb{R}^{1,n}$.

Un résultat classique et bien connu de Geroch montre l'équivalence entre l'existence d'une hypersurface de Cauchy (compacte) et celle d'une fonction temps de Cauchy (propre); ceci implique qu'un espace-temps $GH$ admet une décomposition lisse $M\approx S\times I$ où $I$ est un intervalle de $\mathbb{R}$. Un plongement isométrique dans un autre espace-temps est dit de Cauchy s'il existe une hypersurface de Cauchy dont l'image est une hypersurface de Cauchy. Un espace-temps $(M,g)$ globalement hyperbolique plat est dit maximal si tout plongement de Cauchy dans un autre espace-temps $(N,h)$ globalement hyperbolique plat est surjectif. Dans ce cas on écrit $MGHC$.

Depuis les travaux de Mess \cite{mess1}, Barbot \cite{MR2110829}, et Bonsante  \cite{bonsantethese}, \cite{MR2170277}, \cite{MR2499272}, les espaces-temps $MGHC$ plats sont bien connus y compris en dimension quelconque. Quitte à inverser l'orientation temporelle, on peut les supposer futur-complets, au sens o\`u toutes les g\'eod\'esiques de type temps orient\'ees vers le futur sont compl\`etes. Les espaces-temps $MGHC$ plats futur-complets de dimension $3$ sont caractérisés par la donnée d'une surface riemannienne hyperbolique $S$ et d'une lamination géodésique mesurée $\lambda$ sur $S$. Plus précisément, on associe à chaque  lamination géodésique mesurée $(S, \lambda)$ un espace-temps $MGHC$ de dimension trois $M(S,\lambda)$ dont les surfaces de Cauchy sont homéomorphes à $S$, et tout espace-temps plat $MGHC$ futur-complet de dimension trois et de surface de Cauchy de genre $\geq 2$ est obtenu de cette manière.

Dans le présent travail,  on s'intéresse aux fonctions temps sur un espace $M(S,\lambda)$ \textit{quasi-concaves,} \textit{i.e.} dont les niveaux sont convexes, au sens o\`u leurs futurs sont g\'eod\'esiquement convexes. Un exemple important est celui du temps cosmologique, défini de la manière suivante:  Le temps cosmologique d'un point $p$ est le supremum des longueurs des courbes temporelles passées issues de $p$. D'autres exemples sont le temps $CMC$ et le $K$-temps, c'est-\`a-dire les fonctions temps dont les niveaux
sont \`a courbure moyenne constante pour la premi\`ere, \`a courbure de Gauss constante pour la seconde\footnote{L'existence et l'unicit\'e de ces fonctions temps particuli\`eres  sont assurées par les travaux  de F. B\'eguin, T. Barbot et A. Zeghib \cite{2006math......4486A}, \cite{2008arXiv0804.1053B}.}. Un premier résultat est obtenu par Benedetti et Guadagnini dans \cite{MR1857817}, où ils affirment que la singularité initiale (au sens du temps cosmologique \cite{MR1606594}, \cite{MR2170277}) définie comme étant l'ensemble des points du bord par lesquels passe au moins un hyperplan support de type espace, est un arbre  réel dual à la lamination $\lambda$ (cf. section~\ref{sec:singu}). Nous étudions ici le comportement asymptotique des hypersurfaces de Cauchy vues non comme des variétés riemanniennes mais plutôt comme des espaces de longueurs uniquement géodésiques munis d'une action d'un groupe d'isométrie. Bonsante \cite{MR2170277} montre que les niveaux du temps cosmologique convergent au sens de la topologie spectrale vers la singularité initiale. Dans le cas où la lamination est une multicourbe L. Andersson \cite{MR2216148} a obtenu un résultat analogue pour le temps $CMC$. On généralise ces résultats aux  temps quasi-concaves quelconques.
Le résultat principal est le théorème suivant:

\begin{theo}\label{aa}

Soit $M(S, \lambda)$ un espace-temps $MGHC$ plat non élémentaire futur-complet. Soit $T$ une fonction temps de Cauchy quasi-concave $C^{2}$. Alors les niveaux de $T$ convergent
pour la topologie de Hausdorf Gromov équivariante vers un arbre réel dual à $(S, \lambda)$. En particulier la limite ne dépend pas de la fonction temps choisie.

\end{theo}

\section{Arbres réels dual à une lamination}

Un arbre réel est un espace géodésique $(T,d)$ dans lequel deux points quelconques sont joints par un unique arc, et cet arc est isométrique à un segment. On peut le voir aussi comme étant un espace uniquement géodésique tel que si $[x,y]$ est la géodésique reliant $x$ à $y$, alors tout point $z\in]x,y[$ déconnecte $x$ de $y$. Soit $\Gamma$ un groupe de type fini agissant par isométries sur $T$. Un élément $\gamma$ est dit \textit{elliptique} s'il a un point fixe. Il est dit \textit{hyperbolique} si $l_{T}(\gamma)=\inf_{x\in T}\left\{ d(x,\gamma.x) \right\}$ est atteint par un élément $x$. Un axe de translation d'un élément $\gamma$ est un segment isométrique à $\mathbb{R}$ sur lequel $\gamma$ agit par translation. Morgan et Shalen (\cite{MR769158}) ont montré que l'ensemble $A_{\gamma}=\left\{x\in T \mbox{~tel que~} d(x,\gamma.x)=l_{T}(\gamma)\right\}$ est un axe de translation  si et seulement si $\gamma$ est hyperbolique. L'action de $\Gamma$ sur $T$ est \textit{minimale} si le seul convexe $\Gamma$ invariant non vide est $T$. Elle est \textit{irréductible} s'il n'existe pas de bout fixé par $\Gamma$.

\begin{prop}{\em \cite[Proposition~4.2]{MR1007101}}.
Soit $T$ un arbre réel muni d'une action isométrique  de $\Gamma$. Alors, l'action de $\Gamma$ est minimale si et seulement si $T$ est l' union des axes de translations des \'el\'ements hyperboliques de $\Gamma$.\fin

\end{prop}

\begin{lemm}{\em \cite[Lemme~4.3]{MR1007101}}
\label{le:unionaxe}
Tout segment compact d'un arbre réel minimal irréductible est contenu dans un axe de translation. \fin

\end{lemm}

Un exemple intéressant est celui de l'arbre réel associé à une lamination géodésique mesurée. Soit $\lambda = (\mathcal{L},\mu)$ une lamination géodésique mesurée sur une surface hyperbolique compacte $S$. Relevons cette lamination à $\tilde{S}\approx \mathbb{H}^{2}$ et considérons la pseudo-distance définie par $d_{(\mathcal{L},\mu)}(x,y)$ la borne inférieure des $\mu(\alpha)$, où $\alpha$ est une courbe transverse par morceaux reliant  $x$ à $y$. Notons  $T(\mathcal{L},\mu)$  \textit{l'espace métrique quotient} de $\tilde{S}$ i.e le quotient $\tilde{S}$ par la relation d'équivalence définie par $x\mathfrak{R}y$ si et seulement si $d_{\mathcal{L},\mu}(x,y)=0$, muni de la distance $\bar{d}_{(\mathcal{L},\mu)}$ induite par $d_{(\mathcal{L},\mu)}$.

\begin{prop}{\em \cite{MR1428059}}
\label{pro:dualminimal}
$(T(\mathcal{L},\mu), \bar{d}_{(\mathcal{L},\mu)})$ est un arbre r\'eel minimal et irr\'eductible. \fin

\end{prop}

\section{Singularité initiale}
\label{sec:singu}
Soit $M$ un espace-temps $MGHC$ plat de dimension $2+1$. Nous avons d\'ej\`a rappel\'e dans l'introduction que $M$ est un espace-temps $M(S,\lambda)$ caract\'eris\'e par une lamination g\'eod\'esique mesur\'ee $(S, \lambda)$.
Il est aussi connu (cf \cite{mess1}) que $M$ s'identifie au quotient d'un domaine convexe $\Omega$ de Minkowski par un groupe d'isométrie $\Gamma\subset SO(1,2)\ltimes \mathbb{R}^{1,2}$ agissant librement et proprement discontinuement sur $\Omega$. De plus dans ce cas le temps cosmologique est un temps de Cauchy $C^{1}$ et concave.

Dans toute la suite on se place dans le cas de la dimension trois. On suppose que $M$ est non élémentaire c.\`a.d admet une hypersurface de Cauchy de genre $\geq 2$. Soient $\tau:\Omega\rightarrow]0,+\infty[$ la fonction temps cosmologique de $\Omega$ et $(\Phi^{t}_{\tau})$ son gradient associé. Soit $\tilde{S}_{a}=\tau^{-1}({a})$ les niveaux de $\tau$ pour tout réel strictement positif $a$. La restriction de la métrique lorentzienne globale à un niveau $\tilde{S}_{a}$ induit une métrique riemannienne qu'on note $g_{a}$. On note par $d^{\tau}_{a}$ la distance riemannienne associée à $g_{a}$. Soit $X_{\tau}$ l'espace des lignes du gradient cosmologique. Par la définition même de  l'hyperbolicité globale chaque ligne du gradient intersecte chaque surface de Cauchy de $\Omega$, en particulier, chaque $\tilde{S}_{a}$. Ceci induit une identification naturelle entre $X_{\tau}$ et chaque $\tilde{S}_{a}$ : celle qui envoie chaque ligne du gradient  sur son intersection avec $\tilde{S}_{a}$.  Les distances $d^{\tau}_{a}$ définissent alors une famille à un paramètre de distances sur $X_{\tau}$ qu'on note aussi par abus de notation $d^{\tau}_{a}$. Nous d\'efinissons ici la singularité initiale $\Sigma\subset\partial\Omega$ comme \'etant l'ensemble des points limites des lignes du gradient cosmologique (ceci est coh\'erent avec la d\'efinition donn\'ee pr\'ec\'edemment au vu de \cite{MR2170277}). Ceci définit la rétraction $r:X\rightarrow\partial\Omega$ qui à une ligne associe son point limite dans $\partial\Omega$.

\begin{prop}{\em \cite[Proposition~7.2]{MR2170277}}
Soit $X_{\tau}$ l'espace des lignes de gradient du temps cosmologique. Alors les distances $d^{\tau}_{a}$ convergent, quand $a$ tend vers $0$, pour la topologie compacte ouverte, vers une pseudo-distance $d^{\tau}_{0}$ vérifiant $d^{\tau}_{0}(x,y)=0$ si et seulement si $r(x)=r(y)$.\fin

\end{prop}

Notons $\hat{X}_\tau$ le cleaning de $X_\tau$ pour la pseudo-distance $d^\tau_0$.

Par ailleurs, on peut consid\'erer sur le bord $\partial\Omega$ du domaine $\Omega$ la pseudo-distance $d_\Omega$ qui \`a tout $x$, $y$ dans $\partial\Omega$ associe l'infimum des
longueurs\footnote{Les hyperplans support du convexe $\Omega$ sont tous de type espace ou d\'eg\'en\'er\'es,
les vecteurs tangents \`a une telle courbe sont donc bien de norme positive ou nulle.} pour la m\'etrique de Minkowski des courbes Lipschitz reliant $x$ \`a $y$. F. Bonsante
a d\'emontr\'e que tout point de $\partial\Omega$ est \`a pseudo-distance nulle d'un point de $\Sigma$ (cf \cite{MR2170277}) ; et que la pseudo-distance entre deux points distincts de $\Sigma$ est toujours non-nulle (cf \cite{MR2170277}).
Il y a donc une isom\'etrie entre le cleaning de $(\partial\Omega, d_\Omega)$ et $(\Sigma, d_\Sigma)$, o\`u $d_\Sigma$ d\'esigne la restriction de $d_\Omega$ \`a $d_\Sigma$, et
F. Bonsante a en fait d\'emontr\'e que $(\hat{X}_\tau, d^\tau_0)$ est isom\'etrique \`a $(\Sigma, d_\Sigma)$.

Enfin, rappelons que $M$ est caract\'eris\'e par une lamination g\'eod\'esique mesur\'ee $(S, \lambda)$.

\begin{prop}\label{bb}{\em \cite[Proposition~3.7.2]{MR2499272}}.
$(\hat{X}_\tau, d^\tau_0)$ et $(\Sigma, d_\Sigma)$ sont isom\'etriques \`a l'arbre réel $T(S, \lambda)$. \fin
\end{prop}

\section{Fonction temps quasi-concave}

\begin{defi}

Soit $T:\Omega\rightarrow \mathbb{R}_{+}$ une fonction temps de Cauchy $\Gamma$-invariante. Si les niveaux $T^{-1}({a})$ de $T$ sont convexes alors on dit que $T$ est quasi-concave.

\end{defi}

Dans le cas plat c'est équivalent à dire que le seconde forme fondamentale est positive i.e $\Pi(X,X)=\left\langle \nabla_{X}\mathfrak{n},X\right\rangle\geq 0$ où $\mathfrak{n}$ est la normale aux niveaux de $T$ orient\'ee vers le futur.

Soit $T:\Omega\rightarrow \mathbb{R}_{+}$ une fonction temps quasi-concave $\Gamma$ invariante de classe  $C^{2}$. En reparam\'etrant au but si besoin, on peut toujours supposer, ce qui sera notre convention, que $T$ prend toutes les valeurs dans $\mathbb{R}_+$.
Pour tout réel positif $a$, on note par $\tilde{\Sigma}_{a}$ le niveau $\{ T=a \}$ et par $g_a^T$ la m\'etrique 
induite sur $\tilde{\Sigma}_{a}$\footnote{Remarquons, sans que cel\`a importe pour le pr\'esent travail, qu'il d\'ecoule de la quasi-concavit\'e de $T$ et du Th\'eor\`eme Egregium de Gauss, que chaque $(\tilde{\Sigma}_{a}, g_a^T)$ est CAT$(0)$.}. On considère le champ de vecteurs $\mathfrak{\xi}=\frac{\nabla T}{\left|\nabla T\right|^{2}}$. Une vérification facile montre que le flot $\Phi_{T}^{t}$ associé à $\xi$ préserve les niveaux $\tilde{\Sigma}_{a}$ i.e. envoie le niveau $\tilde{\Sigma}_{a}$ sur $\tilde{\Sigma}_{a+t}$. 

\begin{prop}
\label{pro:Texpand}
Soit $\alpha :[a, b]\rightarrow\Omega$ une courbe de type espace contenue dans le passé de $\tilde{\Sigma}_{1}$. Alors  la longueur de la courbe $\alpha$ est inférieure à celle de $\alpha_{1}$ où $\alpha_{1}(s)=\Phi^{1-T(\alpha(s))}(\alpha(s))$ est la projetée de $\alpha$ sur le niveau $\tilde{\Sigma}_{1}$ le long des lignes du gradient cosmologique.

\end{prop}

\begin{proof}
Soit $X$ un champ de vecteurs  tangent à $\tilde{\Sigma}_{1}$. On le prolonge en un champ de vecteurs $\widehat{X}$ via le flot $(\Phi_{T}^{t})$ .

\[\begin{aligned}
\frac{1}{2}\frac{d}{dt}\left|D\Phi^{t}(X)\right|_{t=0}^{2}=\frac{1}{2}\xi.\left|\widehat{X}\right|^{2}&=\left\langle X,\nabla_{\xi}\widehat{X}\right\rangle &
& =\left\langle X, \nabla_{X}\xi\right\rangle &
& =\left\langle X,\nabla_{X}\lambda(\mathfrak{n})\right\rangle &
\end{aligned}\]

avec $\lambda=\frac{1}{\sqrt{-\left|\nabla T\right|^{2}}}$. On en déduit :

\[\begin{aligned}
\frac{1}{2}\frac{d}{dt}\left|D\Phi^{t}(X)\right|_{t=0}^{2} &= \left\langle X,\lambda\nabla_{X}\mathfrak{n}+X(\lambda)\mathfrak{n}\right\rangle &
&=\lambda\left\langle X,\nabla_{X}\mathfrak{n}\right\rangle &
&=\lambda\Pi(X,X)\geq 0 &
\end{aligned}\]

Ceci implique que si $\alpha_{a}$ est une courbe incluse dans $\tilde{S}_{a}$ alors $\Phi^{t}(\alpha_{a})$ est de longueur plus grande que celle de $\alpha_{a}$.
Soit maintenant $\alpha$ une courbe dans le passé de $\tilde{S}_{1}$. Soit
$\alpha_{1}(s)=\Phi^{1-T(s)}(\alpha(s))$, on sait que $\dot{\alpha}(s)=-T'(s)\xi_{\alpha(s)}+d_{\alpha_{1}(s)}\Phi^{1-T(s)}.\alpha_{1}(s)$. Par suite $\left|\dot{\alpha}(s)\right|^{2}\leq\left|\dot{\alpha}_{1}(s)\right|^{2}$ ce qui donne le résultat voulu.

\end{proof}

\begin{rmq}
\label{rk:cosmoexpand}
Bien que le temps cosmologique ne soit pas $C^2$, cette proposition reste vraie lorsque le temps quasi-concave est le temps cosmologique $\tau$, voir \cite[Proposition~8.1]{2008arXiv0804.1053B} .

\end{rmq}

Une conséquence immédiate est la proposition suivante:

\begin{prop}
Soit $X_{T}$ l'espace des lignes du champs de vecteurs $\xi$ qu'on identifie topologiquement à $\tilde{\Sigma}_{1}$. Alors $d_{a}^{T}$ converge pour la topologie compacte ouverte de $X_T$ vers une pseudo-distance $d_{0}^{T}$.
\end{prop}

\begin{proof}
Pour tout $x$, $y$ dans $X_T$, la famille $d_a^T(x,y)$ d\'ecroit quand $a \to 0$, donc admet une limite $d_0^T(x,t)$. Pour montrer la convergence uniforme,
on reprend les  arguments de   la proposition~ 7.1 dans \cite{MR2170277}. D'apr\`es la Proposition~\ref{pro:Texpand}, pour tout $x$, $y$, $x'$, $y'$ dans $Y$ on a:
$$\left|d^{T}_{a}(x,y)-d^{T}_{a}(x',y')\right|\leq d^{T}_{a}(x,x')+d^{T}_{a}(y,y') \leq d^{T}_{1}(x,x')+d^{T}_{1}(y,y')$$
La  famille $(d^{T}_{a})_{a<1}$  est donc une famille équicontinue d'applications de l'espace m\'etrique produit $(X_T \times X_T, d^T(.,.) + d^T(.,.))$ vers $\mathbb R$. Comme la restriction de $d^{T}_{1}$ \`a $K \times K$  pout tout compact $K$ de $X_T$
est uniform\'ement major\'ee,
cette famille $(d^{T}_{a})_{a<1}$  est aussi localement bornée. La proposition d\'ecoule alors du théorème d'Ascoli Arzéla.
\end{proof}

Comme chaque ligne du gradient du temps cosmologique intersecte chaque niveau du temps $T$ en un unique point, on peut aussi identifier l'espace $X_\tau$ des lignes de gradient du temps cosmologique $\tau$ avec chaque niveau $\tilde{\Sigma}_{a}$ de $T$.
Ceci permet de définir une autre famille de distances sur l'espace $X_{\tau}$. On les note  $\delta_{a}^{T}$. On définit de manière similaires des distances $\delta_{a}^{\tau}$ sur l'espace $X_{T}$. Sous ces hypoth\`eses, on a le résultat suivant

\begin{prop}
\label{pro:compare}
Il existe une suite extraite $(\delta_{a_n}^{\tau})_{n \in \mathbb N}$ (resp $(\delta_{a_n}^{T})_{n \in \mathbb N}$ ) convergeant  vers une pseudo-distance $\delta^{\tau}_{0}$ (resp $\delta^{T}_{0}$) sur $X_T$ (resp $X_\tau$) pour la topologie uniforme sur les compacts. De plus, $$\delta^{\tau}_{0}\leq d^{T}_{0}$$ $$\delta^{T}_{0}\leq d^{\tau}_{0}$$
\end{prop}

\begin{proof}

Soit $K$ un compact de $X_{\tau}$ et soit $a>0$. On sait que $K$ s'identifie à un compact de $\tilde{\Sigma}_{a}$. Donc il existe un $a_{0}>1$ tel que $\tau(K)<a_{0}$. On a que  $\delta^{T}_{a}(x,y)\leq d^{\tau}_{a_{0}}(x,y)$ pour tout $x$ et $y$ dans $K$. Donc la famille $(\delta^{T}_{a})_{a<1}$ est localement bornée.

D'autre part, comme $$\left|\delta^{T}_{a}(x,y)-\delta^{T}_{a}(x',y')\right|\leq \delta^{T}_{a}(x,x')+\delta^{T}_{a}(y,y')$$
on a  $$\left|\delta^{T}_{a}(x,y)-\delta^{T}_{a}(x',y')\right|\leq d_{a_0}^{\tau}(x,x')+ d_{a_0}^{\tau}(y,y')$$ pour tout $x$, $y$, $x'$, $y'$ dans $X_{\tau}$.
D'où l'équicontinuité de la famille $(\delta^{T}_{a})_{a<1}$. On en déduit que $(\delta^{T}_{a})_{a<1}$ admet une valeur d'adhérence $\delta^{T}_{0}$ pour la topologie compacte ouverte.

De plus, on sait que pour tout $\alpha>0$, il existe $a_{\alpha}>0$ tel que les niveaux  $\tilde{\Sigma}_a$ pour $a<a_{\alpha}$ soient dans le pass\'e du niveau cosmologique $\tilde{S}_a$, ce qui implique $\delta^{T}_{a}(x,y)\leq d_{\alpha}^{\tau}(x,y)$.
Lorsque $\alpha$ tend vers $0$, les temps $a_\alpha$ tendent vers $0$; on a donc $\delta^{T}_{0}\leq d^{\tau}_{0}$. 

La démonstration pour  les $\delta_{a}^{\tau}$ est similaire.
\end{proof}

Rappelons que le \emph{spectre} d'un espace m\'etrique $(X,d)$ muni d'une action par isom\'etries d'un groupe $\Gamma$ est l'application qui \`a tout \'el\'ement $\gamma$ de $\Gamma$ associe $\mbox{inf}_{x \in X} d(x, \gamma.x)$.

Une conséquence immédiate de la Proposition~\ref{pro:compare}:

\begin{cor}
\label{cor:spectre}
Les spectres $l_{0,\tau}$, $l_{0,T}$, $l'_{0,\tau}$, $l'_{0,T}$ spectres respectifs de $(X_\tau, d^{\tau}_{0})$, $(X_T, d^{T}_{0})$, $(X_T, \delta^{\tau}_{0})$ et $(X_\tau, \delta^{T}_{0})$ munis de l'action isom\'etrique de $\Gamma$ sont deux-\`a-deux \'egaux.
\end{cor}

\begin{proof}
Soit $\gamma$ un \'el\'ement de $\Gamma$.
D'apr\`es la Proposition~\ref{pro:compare}, on  $\delta^{T}_{0}\leq d^{\tau}_{0}$. Donc, pour tout \'el\'ement $x$ de $X_\tau$, on a
$l'_{0,T}(\gamma) \leq \delta^{T}_{0}(x,\gamma.x)\leq d^{\tau}_{0}(x,\gamma.x)$. Comme $x$ est arbitraire, on obtient $$l'_{0,T}(\gamma) \leq l_{0,\tau}(\gamma)$$

D'autre part, pour tout $x$ dans  $X_\tau$, et tout $a>0$, la distance $\delta^{T}_{a}(x,\gamma.x)$ est la distance $d^{T}_{a}(x_{a},\gamma.x_{a})$ entre $x_a$ et $\gamma.x_a$ dans le niveau $\tilde{\Sigma}_a$, o\`u $x_a$ est l'intersection $x_{a}$ entre $\tilde{S}_a$
et la ligne de gradient cosmologique $x$. Ceci implique que $\delta^{T}_{a}(x,\gamma.x)\geq l_{a,T}(\gamma) \geq l_{0,T}(\gamma)$. Comme $a$ est arbitraire, on a donc $\delta^{T}_{0}(x,\gamma.x)\geq l_{0,T}(\gamma)$. On en déduit:
$$l_{0,T}(\gamma)\leq l'_{0,T}(\gamma)\leq l_{0,\tau}(\gamma)$$
De manière analogue, en \'echangeant les r\^oles de $\tau$ et $T$, on montre:
$$l_{0,\tau}(\gamma)\leq l'_{0,\tau}(\gamma)\leq l_{0,T}(\gamma)$$
Ce qui donne l'égalité $l_{0,T}(\gamma)=l_{0,\tau}(\gamma)$.
\end{proof}

\textit{D\'emonstration du Théorème \ref{aa}. }
Remarquons tout d'abord que la convergence uniforme sur les compacts entraine la convergence de Hausdorff-Gromov équivariante \cite{MR958589}. Pour d\'emontrer le th\'eor\`eme, nous allons d\'emontrer que toutes les limites au sens de la topologie uniforme
sur les compacts $\delta_0^T$ de suites extraites de la famille $(\delta_a)^T_{a>0}$ sont \'egales \`a la pseudo-distance $d_0^\tau$ sur $X_\tau$.
D'apr\`es la Proposition \ref{bb}  on sait que les $(X_\tau, d^{\tau}_{a})$ convergent pour la topologie uniforme sur les compacts vers $(X_\tau, d_0^\tau)$ de sorte que le cleaning $(\hat{X}_\tau, d_0^\tau)$ soit isom\'etrique \`a l' arbre réel $(T,d)$ dual à la lamination géodésique mesurée $(S, \lambda)$. D'apr\`es la Proposition~\ref{pro:dualminimal}, $(T,d)$ est minimal et irréductible. Soient $x$, $y$ deux \'el\'ements de $X_\tau$. D'apr\`es le Lemme~\ref{le:unionaxe} il existe un \'el\'ement $\gamma$ de $\Gamma$ dont l'axe de translation
$A_{\gamma}$ contient le segment $[\bar{x}, \bar{y}]$, o\`u $\bar{x}$, $\bar{y}$ sont les projet\'es de $x$, $y$ dans le cleaning. Quitte \`a remplacer $\gamma$ par une de ses puissances, on peut supposer que $[\bar{x}, \bar{y}]$ est contenu dans le segment $[\bar{x}, \gamma.\bar{x}]$. Or, d'apr\`es la Proposition~\ref{pro:compare}, et comme $\bar{x}$ est dans $A_\gamma$:
$$l_{0,T}(\gamma)\leq\delta^{T}_{0}(\bar{x},\gamma.\bar{x})\leq d^{\tau}_{0}(\bar{x},\gamma.\bar{x})=l_{0,\tau}(\gamma)$$
D'apr\`es le Corollaire~\ref{cor:spectre}, les deux termes aux extr\'emit\'es de cette suite d'in\'egalit\'es sont \'egaux, donc toutes les in\'egalit\'es interm\'ediares sont \'egales, En particulier:
$$\delta^{T}_{0}(\bar{x},\gamma.\bar{x}) = d^{\tau}_{0}(\bar{x},\gamma.\bar{x})$$
Donc:
$$d^\tau_0(\bar{x}, \bar{y}) + d^\tau_0(\bar{y}, \gamma.\bar{x}) = d^\tau_0(\bar{x}, \gamma.\bar{x}) = \delta^{T}(\bar{x}, \gamma.\bar{x}) \leq \delta^{T}_0(\bar{x}, \bar{y}) + \delta^{T}_0(\bar{y}, \gamma.\bar{x})$$
ce qui donne comme voulu, en utilisant \`a nouveau la Proposition~\ref{pro:compare}, l'\'egalit\'e $\delta^{T}_{0}(x,y)=d^{\tau}_{0}(x,y)$.	
\fin

\nocite{*}

\bibliographystyle{plain}

\bibliography{bibliographietest1}

\end{document}